\documentclass[12pt]{amsart}
\usepackage{amsmath}

\calclayout
\usepackage{mathtools, amsmath, amssymb, amsthm, breqn, mathrsfs, hyperref}
\usepackage{xcolor}
\usepackage{lineno}

\usepackage[utf8]{inputenc}
\usepackage[english]{babel}
\usepackage{csquotes}
\usepackage{comment}

\usepackage[backend=biber]{biblatex}
\author{Harshith Alagandala}
\date{\today}

\newtheorem{theorem}{Theorem}[section]

\newtheorem{lemma}[theorem]{Lemma}

\theoremstyle{definition}
\newtheorem{definition}[theorem]{Definition}
\newtheorem*{remark}{Remark}

\newtheorem{example}{Example}[theorem]

\newcommand{\C}{\mathbb{C}}
\newcommand{\R}{\mathbb{R}}

\newcommand{\norm}[1]{\left\|#1\right\|}

\renewcommand{\Re}{\text{Re}}
\renewcommand{\Im}{\text{Im}}

\title{Local polynomial convexity at hyperbolic CR-singularities in $M\subset \C^n$}
\begin{document}
\maketitle
\tableofcontents

\section{Introduction}

Let $M$ be a real $k$-dimensional smooth submanifold of $\C^n$.
The presence of CR-singularities is often governed by the topology of M:
not every compact surface has a totally real embedding in $\C^2$, e.g.,
a sphere embedded in $\C^2$ will have CR-singularities (\textcite{bishop1965differentiable}).
The structure of CR-singularities depends on Euler characteristic and Pontryagin numbers 
for a generically embedded manifold $M$ (\textcite{lai1972}, \textcite{webster1985}, \textcite{domrin1995}).
When $n < \lfloor \frac{3k}{2} \rfloor$, there are $k$-dimensional manifolds which do not have totally real embedding in $\C^n$ (\textcite{jacobowitz2012}).
Such a manifold will have CR-singularities where it fails to be totally real.

It is of interest to study polynomially convex embeddings of a manifold in $\C^n$.
There are global results for certain dimensions:
a non-maximal ($k<n$) totally real embedding of $M$ in $\C^n$ can be deformed via a small perturbation into a polynomially convex one. % CITE
When $ \lfloor \frac{3k}{2} \rfloor \leq n$, any $k$-dimensional compact real manifold admits a totally real embedding into $\C^n$. 
%This condition is sharp as mentioned above \cite{jacobowitz2012}.
Stronger bounds exist (\textcite{gupta-shafikov2020}, \cite{gupta-shafikov2020-odd}, \cite{gupta-shafikov2025}).
We are interested in the local polynomial convexity of~$M$. 
For a generic embedding $M^k \subset \C^n$ where $k<n$,
we will show in Section \ref{section_local_polyconvex_k_less_n} that
$M$ is locally polynomially convex at a generic CR-singular point $p \in M$ of order 1 (Definition \ref{def:order_of_CR_singularity}).
For $k>n$, if $M$ does not contain any proper submanifolds of the same CR dimension, 
then a family of analytic discs can be attached to $M$ which obstructs polynomial convexity (\textcite{tumanov1989}).
%This causes obstruction to $\mathscr{P}(K) = \mathscr{C}(K)$
%for any compact neighbourhood $K \subset M$ of $p$. % DONE+CITE tumanov + ebenfielt
This paper will focus at the maximal case when $k=n$.
Then $M$ is locally polynomially convex at $p$ when the point is not a CR-singularity (as $M$ is totally real at $p$).
Can we asses the local polynomial convexity at CR-singularities?

When $k=n$ and $n \leq 7$, only order 1 CR-singularities may occur for a generic embedding (\textcite[Proposition 2.1]{webster1985}, \textcite{lai1972}).
When $n \geq 8$, higher order CR-singularities may exist, but order 1 CR-singularities will still be present.
Further, given an order 1 CR-singularity in~$M$, there exists a neighbourhood of that point where the CR-singularities of $M$ are of order~1.

In this paper, we assess the local polynomial convexity of $M^n \subset \C^n$ at order 1 CR-singularities. 
For this, we use the classification given by \textcite{bishop1965differentiable}:
let $M^n \subset \C^n$ be a smooth manifold with an order 1 CR-singularity at $0\in M$.
Then after a biholomorphic change of coordinates, $M$ can be parametrized near $0$ as
\begin{equation}
	\begin{split}
		z_1 &= t_1 + i f_1(t_1, \hdots, t_{n-2}, w, \bar w), \\
				&\vdots \\
		z_{n-2} &= t_{n-2} + i f_{n-2}(t_1, \hdots, t_{n-2}, w, \bar w), \\
		z_{n-1} &= w = u + iv, \\
		z_{n} &= w\bar w + \gamma (w^2 + \bar w^2) + F(t_1, \hdots, t_{n-2}, w, \bar w),
	\end{split}
	\label{eqn_bishop_normal_form}
\end{equation}
where  $(t_1, \hdots, t_{n-2}) \in [-T,T]^{n-2}$ for some $T>0$,
$w \in \overline{B_r(0)} \subset \C$ for some $r>0$,
$f_j:[-T,T]^{n-2} \times \overline{B_r(0)} \to \R$ are smooth functions that vanish to order two at zero for $1 \leq j \leq n-2$,
$F:[-T,T]^{n-2} \times \overline{B_r(0)} \to \C$ is a smooth function that vanishes to order three at zero,
and $\gamma \geq 0$.

Note that $\gamma$ is independent of the biholomorphic coordinate change chosen.
We will refer to equation (\ref{eqn_bishop_normal_form}) as the normal form of an order 1 CR-singularity.
An order 1 CR-singularity $p \in M$ can be classified as elliptic, parabolic or hyperbolic point.
We can bring $M$ to the normal form (\ref{eqn_bishop_normal_form}) centered at $p$.
If $\gamma < \frac{1}{2}$, then $p$ is elliptic.
If $\gamma = \frac{1}{2}$, then it is parabolic.
And if $\gamma > \frac{1}{2}$, then it is hyperbolic.
This classification can also be obtained without the use of a normal form (\textcite{webster1985}).

\textcite{bishop1965differentiable} has shown that $M^n\subset \C^n$ is not locally polynomially convex at an elliptic point
by attaching a family of analytic discs around the point.
For a generic embedding, the set of parabolic points is nowhere dense in the set of CR-singularities.
Further, existence of elliptic points near parabolic points will prevent local polynomial convexity of $M$ at parabolic points.
%They form a co-dimension~1 submanifold of the set of CR-singularities of $M$.
\textcite{forstnerivc1991new} have shown that for a two dimensional submanifold $M^2 \subset \C^2$,
$M$ is polynomially convex at hyperbolic complex points.
In this paper, we will look at local polynomial convexity at hyperbolic points of $M^n \subset \C^n$ when $n\geq 3$.

The first result will impose a condition on $F$ in equation (\ref{eqn_bishop_normal_form}).
\begin{theorem}
  \label{thm_hyperbolic_higher_dimension_vanish_order_two}
	Let $p$ be a hyperbolic point of $M^n \subset \C^n$. 
	Let equation (\ref{eqn_bishop_normal_form}) be the normal form of $M$ centered at $p$ ($p=0$).
	If $F(t,w,\bar w)$ vanishes to order two in $w$ (Definition \ref{defn:vanishes_order_two_in_w}), then $M$ is locally polynomial convex at $p$.
	Further, $\mathscr{P}(K) = \mathscr{C}(K)$ is satisfied on some compact neighbourhood $K \subset M$ of $p$.
\end{theorem}

The second result will impose a condition on $f_j$.
A hyperbolic point $p\in M$ will be called \emph{flat} if $f_j$ in the normal form (\ref{eqn_bishop_normal_form})
do not depend on $t_{j+1}, \hdots, t_{n-2}, w, \bar w$.
Equivalently $M$ can be parametrized centered at $p$ as
\begin{equation}
	\label{eqn_flat_normal_form}
	\begin{split}
		z_1 &= t_1 + i f_1(t_1), \\
		z_2 &= t_2 + i f_1(t_1, t_2), \\
				&\vdots \\
		z_{n-2} &= t_{n-2} + i f_{n-2}(t_1, \hdots, t_{n-2}), \\
		z_{n-1} &= w = u + iv, \\
		z_{n} &= w\bar w + \gamma (w^2 + \bar w^2) + F(t_1, \hdots, t_{n-2}, w, \bar w),
	\end{split}
\end{equation}
where
$(t_1, \hdots, t_{n-2}) \in [-T,T]^{n-2}$,
$w = u+i v \in \overline{B_r(0)}$,
$f_j:[-T,T]^{j} \to \R$ are smooth functions that vanish to order two at zero for $j=1,\hdots,n-2$,
$F:[-T,T]^{n-2} \times \overline{B_r(0)} \to \C$ is a smooth function that vanishes to order three at zero,
and $\gamma \geq \frac{1}{2}$.

\begin{theorem}(Flat hyperbolic point)
  \label{thm_flat_CR_singularity}
	Let $p$ be a hyperbolic point of $M^n \subset \C^n$ which is flat, i.e., $M$ can be represented
	locally at $p$ by the above equation (\ref{eqn_flat_normal_form}).
	Then $M$ is locally polynomially convex at $p$.
	Further, $\mathscr{P}(K) = \mathscr{C}(K)$ is satisfied on some compact neighbourhood $K \subset M$ of $p$.
\end{theorem}

For the proof of the above theorem, we will prove a result of independent interest.
\textcite{forstnerivc1991new} assert the local polynomial convexity of $M^2\subset\C^2$ at hyperbolic points.
We will give a quantitative size for a polynomially convex neighbourhood of a hyperbolic point in $M^2\subset \C^2$
that depends on the normal form (\ref{eqn_defining_equation_of_M_2_in_eqn_form}).

\begin{theorem}
	\label{thm_polynomial_convexity_two_dim_with_C2_condition}
	Let $M \subset \C^2$ be a manifold given by
	\begin{equation}
		z_2 = z_1 \bar z_1 + \gamma (z_1^2 + \bar z_1^2) + F(z_1, \bar z_1),
		\label{eqn_defining_equation_of_M_2_in_eqn_form}
	\end{equation}
	where $\gamma > \frac{1}{2}$,
	$z_1 \in \overline{D_R(0)} \subset \C$,
	and a smooth function $F: \overline{D_R(0)} \to \C$ 
	that vanishes to order three at $0\in M$.

  Choose $r\in(0,R]$ such that
  \begin{equation}
    \label{eqn_theorem-C2-bound-main}
    \norm{F}_{\mathscr{C}^2(\overline{B_r(0)})} 
    \leq \frac{(2\gamma-1)^3}{2^{14} \gamma^3}.
  \end{equation}
	Define $M_r =  M \cap (\overline{B_r(0)} \times \C )$.
  Then $M_r$ is polynomially convex and satisfies
	$ \mathscr{P}(M_r) = \mathscr{C}(M_r) $.
\end{theorem}

It is evident that the size of the neighbourhood decreases as $\gamma$ approaches $\frac{1}{2}$.
Conversely, if $F \equiv 0$, then the size of the neighbourhood can be arbitrarily large.

Theorem \ref{thm_hyperbolic_higher_dimension_vanish_order_two} and Theorem \ref{thm_flat_CR_singularity}
do not cover all possible hyperbolic points.
The following is an example of a hyperbolic point that does not satisfy the condition for both theorems:
set $\gamma > \frac{1}{2}$ and
consider the manifold $N^3 \subset \C^3$ given by the parametric equation:
\begin{equation*}
	\begin{split}
		z_1 =& t+ i(w^2 + \bar w^2), \\
		z_2 =& w, \\
		z_3 =& w\bar w + \gamma(w^2 + \bar w^2) + t^2 (\bar w + w), \\
	\end{split}
\end{equation*}
where $t\in [-T,T]$ and $w \in \overline{B_r(0)}$ where $T,r>0$.
This manifold can also be represented as the solution to equations:
\begin{equation}
	\begin{split}
		y_1 =& z_2^2 + \bar z_2^2, \\
		z_3 =& z_2\bar z_2 + \gamma(z_2^2 + \bar z_2^2) + x_1^2 (\bar z_2 + z_2), \\
		y_3 =& 0. \\
	\end{split}
\end{equation}
The point $0\in N$ is a hyperbolic point. We do not know whether $N$ is locally polynomially convex at $0$.
% DONE: EXample

% DONE: Outline of paper
In Section \ref{section_background}, we will recall machinery to assess polynomial convexity,
review CR-singularities, and prove some technical inequalities.
In Section \ref{section_klessn}, we will investigate the local polynomial convexity of $M^k \subset \C^n$ for $k<n$
by making use of the \textcite{coffman2006} normal form.
Then we proceed to prove Theorem \ref{thm_hyperbolic_higher_dimension_vanish_order_two} in Section \ref{section_proof_of_thn_vanish_order_two}.
In Section \ref{section_hyperbolic_m2_c2}, we will give a quatitative estimate on the size of 
polynomially convex neighbourhood at a hyperbolic point of $M^2 \subset \C^2$ (Theorem \ref{thm_polynomial_convexity_two_dim_with_C2_condition}).
Finally, in Section \ref{section:flat-hyperbolic-point}, we will prove local polynomial convexity at flat hyperbolic points (Theorem \ref{thm_flat_CR_singularity}).

% DONE: Tanks rasul
I would like to express my deepest gratitude to Professor Rasul Shafikov, whose guidance, knowledge, and support were instrumental in completing this research.

\section{Background}
\label{section_background}

In this section, we gather background material on polynomial convexity, CR-singularities, and prove some technical inequalities.

\subsection{Polynomial Convexity}

\begin{definition}[Polynomial Convexity]
	A compact subset $K \subset \C^n$ is \emph{polynomially convex} if
	for any point $w\in \C^n \setminus K$, there exists a (holomorphic)
	polynomial $P:\C^n \to \C$ such that 
	\[
		|P(w)|>\sup_{z\in K}|P(z)|.
	\]
\end{definition}

Let $\mathscr{C}(K)$ denote the set of continuous functions on a compact subset $K\subset \C^n$.
This set forms a Banach algebra when equipped with the norm $\norm{P}_K := \sup_{z\in K} |P(z)|$.
Define $\mathscr{P}(K)$ as the uniform closure in $\mathscr{C}(K)$ of the set of polynomials restricted to $K$.
The following theorem relates function algebras to polynomial convexity.

\begin{theorem}[{\cite[Theorem 1.2.10]{stout2007polynomial}}]
	A compact subset $K \subset \C^n$ is polynomially convex if $\mathscr{C}(K) = \mathscr{P}(K)$.
\end{theorem}

\begin{definition}[Local Polynomial Convexity]
	A compact subset $K$ is \emph{locally polynomially convex} at $p$ if there exists a compact polynomially convex neighbourhood $U$ of $p$ in $K$.
\end{definition}

% Equivalently, there exists $\{U_\alpha\}_{\alpha\in\Lambda}$ an open neighbourhood basis of $p$ in $\C^n$
% such that $\overline{M\cap U_\alpha}$ is polynomially convex for all $\alpha \in \Lambda$.

\begin{lemma}
	A compact set $K$ is locally polynomially convex at $p$, if and only if, there exists a polynomially convex neighbourhood basis of $p$ in $K$.
\end{lemma}
\begin{proof}
	$(\implies)$
	Take the intersection of $K$ with sufficiently small closed balls.
	$(\impliedby)$ Follows by definition of local polynomial convexity.
\end{proof}

\begin{remark}
	If there exists a compact neighbourhood $K$ of $p$ in $M$ such that $\mathscr{C}(K) = \mathscr{P}(K)$,
	then $M$ is locally polynomially convex at $p$.
\end{remark}

The polynomial convexity of a set can be determined by the polynomial convexity of
slices formed by taking fibers over a real-value function:
\begin{theorem}[{\cite[Theorem 1.2.16]{stout2007polynomial}}]
	\label{thm_fibre_theorem}
	If $X \subset \C^n$ is compact and if $\mathscr{P}(X)$ contains a real-valued function $f$,
	then $X$ is polynomially convex if and only if each fiber $f^{-1}(t),\; t\in\R$, is polynomially convex.
	If $X$ is polynomially convex, then $\mathscr{P}(X) = \mathscr{C}(X)$ if and only if
	$\mathscr{P}(f^{-1}(t)) = \mathscr{C}(f^{-1}(t))$ for each $t \in \R$.
\end{theorem}
\begin{remark}
The set $\R$ can be replaced by any arc in $\C$.
\end{remark}

We know that a compact subset $K$ of $\R^n \subset \C^n$ is polynomially convex
and satisfies $\mathscr{P}(K) = \mathscr{C}(K)$. The same holds if $\R^n$ is replaced by a totally real plane.
The following theorem determinates the polynomial convexity of perturbations of a totally real plane:
\begin{theorem}[{\cite[Corollary 1.6.12]{stout2007polynomial}}]
	\label{thm_lipschitz_graph_general_polynomial_convexity}
	Let $X$ be a compact subset of $\R^n$, let $R=(R_1, ..., R_n): X \to \C^n$
	be a map that satisfies the Lipschitz condition $|R(z) - R(z')| < c|z - z'|$
	for all $z,z'\in X$ and some fixed $c\in [0,1)$.
	Consider the set $X_R$ defined as \[
		X_R = \{
			(z_1, ..., z_n, \bar z_1 + R_1(z), ..., \bar z_n + R_n(z)): z\in X
		\}.
	\]
	Then $\mathscr{P}(X_R) = \mathscr{C}(X_R)$, and $X_R$ is polynomially convex in $\C^{2n}$.
\end{theorem}

\begin{example}[Totally real manifold]
Let $M \subset \C^n$ be a smooth $k$-dimensional manifold for $k<n$.
If $M$ is totally real at $p$, then $M$ is locally polynomially convex at $p$ \cite[Corollary 1.6.15]{stout2007polynomial}.
There exists a compact neighbourhood $K \subset M$ of $p$ which satisfies $\mathscr{C}(K) = \mathscr{P}(K)$.
\end{example}

Arbitrary intersection of polynomially convex sets is polynomially convex.
On the other hand, the polynomial convexity of the union of polynomial convex sets is a difficult question.
We will use the following version of Kallin's lemma {\cite{kallin_lemma}, \cite[Theorem 1.6.19]{stout2007polynomial}}:
\begin{theorem}[Kallin's lemma]
	\label{thm_kallin_lemma}
	Let $K_1$ and $K_2$ be compact polynomially convex sets in $\C^n$.
	Let $P: \C^n \to \C$ be a polynomial such that \[
		P(K_1) \subset \{x+iy \in \C: y > 0 \} \cup \{(0,0)\},
	\] 
	\[
		P(K_2) \subset \{x+iy \in \C: y < 0 \} \cup \{(0,0)\}.
	\]
	If $P^{-1}(0) \cap (K_1 \cup K_2)$ is polynomially convex in $\C^n$,
	then $K_1 \cup K_2$ is polynomially convex.
	Further, if $\mathscr{P}(K_1) = \mathscr{C}(K_1)$ and $\mathscr{P}(K_2) = \mathscr{C}(K_2)$,
	then $\mathscr{P}(K_1\cup K_2) = \mathscr{C}(K_1 \cup K_2)$.
\end{theorem}

Under a holomorphic map $F:\C^n \to \C^n$,
the pre-image of a polynomially convex
set is polynomially convex if the pre-image is compact: 
let $K \subset \C^n$ be polynomially convex.
If $x \in \C^n \setminus F^{-1}(K)$, then $F(x) \in \C^n \setminus K$, and
we have a polynomial $P$ such that $P(F(x))=1$ and $\norm{P}_K < 1$.
The function $P\circ F$ is entire and $P\circ F(x) = 1$ with $\norm{P\circ F}_{F^{-1}(K)} < 1$.
Hence, $F^{-1}(K)$ is polynomially convex.
In addition, if the map is proper, we have \cite[Theorem 1.6.24]{stout2007polynomial}:

\begin{theorem}[Proper holomorphic mapping]
  \label{thm_proper_holo_mapping}
Let $F:\C^n \to \C^n$ be a holomorphic map that is proper. Then $K$ is polynomially convex
if and only if $F^{-1}(K)$ is polynomially convex,
and $\mathscr{P}(K) = \mathscr{C}(K)$ if and only if
$\mathscr{P}(F^{-1}(K)) = \mathscr{C}(F^{-1}(K))$. 
\end{theorem}
A map is called proper if the pre-image of a compact set is compact.
In particular, polynomial convexity is preserved under 
biholomorphic maps from $\C^n$ to $\C^n$.

\subsection{CR-singularities}

Let $M$ be a real $k$-dimensional smooth submanifold of $\C^n$.
At a point $p \in M$, we study the structure of the tangent space $T_pM \subset T_p\C^n$.
%The dimension $\dim_\R T_pM = k$ and $\dim_\R T_p\C^n = 2n$.
Let $J$ be the almost complex structure on $T_p\C^n$ and define the vector space
\[
	H_pM := \{ X \in T_pM : J(X) \in T_pM \} = T_pM \cap JT_pM.
\]
The almost complex structure $J$ turns $H_pM$ into a complex vector space.
Then $H_pM$ is the maximal complex vector space that is a subset of $T_pM$. 
We have $\dim_\R T_pM = k$ and $\dim_\R JT_pM = k$.
By the pigeon-hole principle, $\dim_\C H_pM \geq \max\{k-n, 0\}$.

\begin{definition}
	The \emph{CR-dimension} of $M$ at $p$ is defined as $\dim_\C H_pM$.
	The manifold $M\subset \C^n$ is called a CR-submanifold
	if $p \mapsto \dim_\C H_pM$ is constant on $M$.
	The manifold $M$ is said to be \emph{totally real at a point} $p$ if $\dim_\C H_pM = 0$.
	The manifold $M$ is called \emph{totally real} if $M$ is totally real at $p$ for all $p\in M$.
\end{definition}

For a generic embedding $M \subset \C^n$
and a generic point $p\in M$, the CR-dimension is equal to $\max\{k-n,0\}$.
When $n=k$, the CR-dimension of $M$ is zero at a generic point.

\begin{definition}
	\label{def:order_of_CR_singularity}
	A point $p \in M$ is a CR-singularity of order $\mu \in \mathbb{Z}^+$ if 
	\[
		\dim_\C H_pM = \max\{k-n,0\} + \mu.
	\]
\end{definition}
For a generically embedded $M$ and $k \leq n$, the set of CR-singularities of order $\mu$ given by $\{p\in M: \text{CR-dim}_p M = \mu \}$
forms a (non-closed) submanifold of codimension $2\mu(n-k+\mu)$ in $M$ \cite{domrin1995}.

\subsection{Preliminary inequalities}

Similar to the Bernoulli's inequality for real numbers, we require the following for complex numbers:
\begin{lemma}
  \label{lemma_sqrt_inequality}
	For the principle branch of square root,
	$$
		\left| 1 - \sqrt{1+z} \right| \leq \frac{10}{11} |z|
	$$
	when $|z|<\frac{1}{4}$.
\end{lemma}

\begin{proof}
  Let $z \in \C$ such that $|z| < \frac{1}{4}$, 
  and 
	$
		q(z) =  1 - \sqrt{1+z} 
	$.
	Then $ z = q(z) (q(z)-2) $.

	Let $z = r e^{i\theta} = r (\cos(\theta) + i \sin(\theta))$ where $0\leq r \leq 1/4$ and $\theta \in (-\pi, \pi]$.
  Then $ 1+z = (1+r \cos(\theta)) + i (r \sin(\theta))$.
	Write $1+z$ in polar coordinates as $ 1+z = |1+z| e^{i\phi}= |1+z| (\cos(\phi) + i \sin(\phi))$ where $\phi \in (-\pi, \pi]$.

	We can compute $|1+z|^2 = | (1+r \cos(\theta)) + i (r \sin(\theta)) |^2 = \sqrt{1+r^2 + 2r \cos(\theta)}$.
	This yields the inequality
	\begin{equation}
		\label{eqn_inequality_norm_1+z}
			3/4 < (1-r) \leq |1+z| \leq (1+r) \leq 5/4.
	\end{equation}
	We get $$
	\phi = \arctan \left( \frac{r\sin(\theta)}{1+r\cos(\theta)} \right)
	= \arccos \left( \frac{1+r\cos(\theta)}{\sqrt{1+r^2+2r\cos(\theta)}} \right),
	$$
	and $\sqrt{1+z} =\sqrt{|1+z|} (\cos(\phi/2) + i \sin(\phi/2))$.
	As $r \leq 1/4$, we can show
	\begin{equation}
		\label{eqn_inequality_phi_2}
	\cos(\phi/2) \geq \sqrt{\frac{4}{5}}.
	\end{equation}

	We simplify $\left|1-\sqrt{1+z} \right|^2 = 1 + |1+z| - 2\sqrt{|1+z|}\cos(\phi/2)$.

	From equation (\ref{eqn_inequality_norm_1+z}) and (\ref{eqn_inequality_phi_2}), we have
	\begin{gather*}
		\left|1-\sqrt{1+z} \right|^2  < \frac{4}{5}.
  \end{gather*}
  Hence,
  $ |q(z)| < 9/10$.

	Using this inequality with $z = q(z) (q(z)-2)$:
	$$
	|q(z)| \leq \frac{10}{11} |z|
	$$
	when $|z| \leq \frac{1}{4}$.

\end{proof}

In then following lemmas, $F_w$ and $F_{\bar w}$ denote the partial derivatives of $F$ with respect to $w$ and $\bar w$ respectively.

\begin{lemma}
  \label{lemma_C2_inequality}
	If $F: D_r(0) \to \C$ is a smooth function, then
	$$
	\left| \frac{F_{w}(w')}{w'} \right| \leq 4 \norm{F}_{\mathscr{C}^2}; \quad
	\left| \frac{F_{\bar w}(w')}{w'} \right| \leq 4 \norm{F}_{\mathscr{C}^2}; \quad
	\left| \frac{F(w')}{(w')^2} \right| \leq 8 \norm{F}_{\mathscr{C}^2}
	$$
	for $w'\in D_r(0)$.
\end{lemma}
\begin{proof}
  By the Mean Value Theorem for real-valued function $\Re F_x$, we have a point $w_o$ on the line segment
  joining $w'$ and $0$ such that
  \[
    \Re F_{x}(w') - 0 = \Re F_{xx}(w_o) \Re(w' - 0) + \Re F_{xy}(w_o) \Im(w'-0),
  \]
  \[
    |\Re F_{x}(w')|  \leq |F_{xx}(w_o)| |w'| + |F_{xy}(w_o)| |w'| \leq 2 \norm{F}_{\mathscr{C}^2} |w'|.
  \]
  Similarly, we get
  \[
    |\Im F_{x}(w')| \leq 2 \norm{F}_{\mathscr{C}^2} |w'|.
  \]
  Hence,
  \[
    |F_{x}(w')| = |\Im F_{x}(w')| + |\Re F_{x}(w')| \leq 4 \norm{F}_{\mathscr{C}^2} |w'|
  \]
  and
  \[
    |F_{w}(w') | \leq \frac{1}{2}|F_{x}(w')| + \frac{1}{2}|F_{y}(w')|\leq 4\norm{F}_{\mathscr{C}^2} |w'|.
  \]
  Similarly,
  \[
    |F_{\bar w}(w') | \leq 4\norm{F}_{\mathscr{C}^2} |w'|.
  \]

  For the final inequality, apply then Mean Value Theorem on the real function $\Re F$:
  we have a point $v$ on the line segment joining $w'$ and $0$ such that
  \[
    \Re F(w') - 0 = \Re F_x(v) \Re(w' - 0) + \Re F_y(v) \Im(w' - 0).
  \]
  By taking the absolute value, we get
  \[
    |\Re F(w')| \leq |\Re F_x(v)| |w'| + |\Re F_y(v)| |w'|.
  \]
  As shown before, $|\Re F_x(v)| \leq 2 \norm{F}_{\mathscr{C}^2} |v|$.
  Since $v$ is in the line segment joining $w'$ and $0$, we have $|v|\leq |w'|$ and
  \[
    |\Re F(w')| \leq 4 \norm{F}_{\mathscr{C}^2} |w'|^2.
  \]
  Hence,
  \[
    |F(w')| = |\Re F(w')| + |\Im F(w')|  \leq 8 \norm{F}_{\mathscr{C}^2} |w'|^2.
  \]

\end{proof}

\begin{lemma}
  \label{lemma_complex_meanvalue_theorem_lipschitz}
  Let $f:D_r(0) \to \C$ be a $\mathscr{C}^1$-smooth function and $r,k>0$.
If $|f_w(w_o)| \leq k$ and $|f_{\bar w}(w_o)| \leq k$ for all $w_o \in D_r(0)$,
  then \[|f(w') - f(w'')| \leq 8k |w' - w''|\]
  for all $w',w'' \in D_r(0)$.
\end{lemma}
\begin{proof}
	We have
  \[
    |f_x(w_o)| = |f_w(w_o) + f_{\bar w}(w_o)| 
    \leq |f_w(w_o)| + |f_{\bar w}(w_o)| \leq 2k.
  \]
  Similarly, $|f_y(w_o)| \leq 2k$.
  By the Mean Value Theorem for real valued function, we have a point $w_o$ on the line segment
  joining $w'$ and $w''$ such that
  \[
    \Re f(w') - \Re f(w'') = \Re f_x(w_o) \Re (w' - w'') + \Re f_y(w_o) \Im (w' - w'').
  \]
  Taking the absolute value gives
  \[
    |\Re f(w') - \Re f(w'')| \leq |\Re f_x(w_o) \Re (w' - w'')| + |\Re f_y(w_o) \Im (w' - w'')|
  \]
  \[
    \leq |f_x(w_o)| |(w' - w'')| + |f_y(w_o)| |(w' - w'')| \leq 4 k |w' - w''|.
  \]
  Similarly,
  \[
    |\Im f(w') - \Im f(w'')| \leq 4k |w' - w''|.
  \]
  Now we split the following into real and imaginary parts
  \[
  |f(w') - f(w'')| \leq |\Re f(w') - \Re f(w'')| + |\Im f(w') - \Im f(w'')| \leq 8k |w'-w''|.
  \]

\end{proof}

\section{Local polynomially convexity of $M^k \subset \C^n$ for $k<n$}
\label{section_klessn}

\label{section_local_polyconvex_k_less_n}
When $\lfloor \frac{3k}{2} \rfloor  \leq n$, $M^k$ admits a polynomially convex embedding,
and the set of polynomially convex embeddings are dense in the set of all smooth embeddings 
(see \textcite{low-wold2009}, \textcite{forstneric-rosay1993}).

Let $\frac{2}{3}(n+1) \leq k < n$. In \cite{coffman2006}, Coffman introduced a nondegeneracy condition for order 1
CR-singular points of a generically embedded smooth manifold $M^k \subset \C^n$.
This nondegeneracy holds on a dense open subset of CR-singular point according to \textcite{gupta-shafikov2025}.
At such a point, Coffman constructed a normal form (\cite[Proposition 3.3]{coffman2006}):
when $M$ is real analytic,
there is a local biholomorphic change of coordinates centered at a nondegenrate CR-singular point $p \in M$
where $M$ is given by
	\begin{align}
		\label{eqn_normalform_coffman}
		y_s &= 0 & \text{for } s= 2, ..., k-1,
		\\
		z_t &= \bar z_1 ( z_{2(t-k+2)+1} + i z_{2(t-k+2)+1}) & \text{for } t=k, ..., n-2,
		\nonumber
		\\
		z_{n-1} &= \bar z_1^2, &
		\nonumber
		\\
		z_n &= \bar z_1 (z_1 + x_2 + i x_3).&
		\nonumber
	\end{align}

The following proof is inspired by \textcite[Lemma 3.1]{gupta-shafikov2020}

\begin{theorem}
	Let $M^k \subset \C^n$ be a real analytic manifold and $p\in M$ be a CR-singularity of order 1 that satisfies the
	nondegeneracy condition \cite{coffman2006}. Then $M$ is locally polynomially convex at $p$.
	Further, there exists a compact neighbourhood $K$ of $p$ such that $\mathscr{P}(K) = \mathscr{C}(K)$.
\end{theorem}
\begin{proof}
	We will consider the Coffman normal form (\ref{eqn_normalform_coffman}) for $M$ at centered at $p$. 

	Let $K_\epsilon = M \cap (\overline{B^1_\epsilon(0)} \times \C^{n-1})$ where $\epsilon >0$ is small.
	Consider the map $\tilde G: \C^n \to \C^{k-2}$ as the projection on the coordinates
	$(z_1,..., z_n) \mapsto (z_2, ..., z_{k-1})$. The restriction of $G:=G|_{K_\epsilon}$  takes values in $\R^{k-2}$
	as $y_s = 0$ for $s = 2,..., k-1$.
	By the iterated version of Theorem \ref{thm_fibre_theorem}, we have
	$K_\epsilon$ is polynomially convex if and only if for each 
	$t=(t_2,...,t_{k-1})\in \R^{k-2}$, the fiber $M^2_t := G^{-1}(t_2,...,t_{k-1})$ is polynomially convex.

	We can parametrize $M^2_t$ with one complex variable $\zeta \in \overline{B_\epsilon^1(0)}$.
	Set $z_1 = \zeta$ and fix $z_j = t_j$ for $j=2,...,k-1$ in the normal form (\ref{eqn_normalform_coffman}).
	Let this parametrization be $\phi: \overline{B_\epsilon^1(0)} \to K_\epsilon$.
	Note that $\pi_1 \circ \phi(\zeta) = \zeta$ and
	$\pi_{n-1} \circ \phi(\zeta) = \bar \zeta^2$. Hence, the algebra \[
		\mathcal {A} := \overline { \{ P \circ \phi | P:\C^n \to \C \text{ polynomial }\} }
	\] contains the algebra generated by $\zeta, \bar \zeta^2$ on $\overline{B_\epsilon^1(0)}$.
	For $\epsilon$ small, this algebra equals $\mathscr{C}(\overline{B_\epsilon^1(0)})$  \cite{minsker1976}.
	By pulling back continuous function, we get $\mathscr{C}(M^2_t) = \mathscr{P}(M^2_t)$:
	if $f\in \mathscr{C}(M_t^2)$, then $f \circ \phi \in \mathscr{C}(\overline{B_\epsilon^1(0)})$. So $f \circ \phi \in \mathcal{A}$
	and there is a sequence of polynomials $\{P_n\}$ such that $P_n \circ \phi$ limits to $f\circ \phi$ uniformly on $\overline{B_\epsilon^1(0)}$.
	Hence, $\{P_n\}$ limits to $f$ uniformly on $M^2_t$.
	Which implies $M^2_t$ is polynomially convex for each $t$. Hence, $M$ is polynomially convex.
	Now $\mathscr{P}(K_\epsilon) = \mathscr{C}(K_\epsilon)$ follows from the second part of Theorem \ref{thm_fibre_theorem}.
\end{proof}

\section{Proof of Theorem~\ref{thm_hyperbolic_higher_dimension_vanish_order_two}}
\label{section_proof_of_thn_vanish_order_two}

We will prove Theorem~\ref{thm_hyperbolic_higher_dimension_vanish_order_two} in this section.

\begin{definition}
	\label{defn:vanishes_order_two_in_w}
We say a smooth function $F: [-T,T]^{n-2} \times \overline{B_r(0)}$ vanishes to order two in $w$ if
\[
	\lim_{\substack{(t,w)\to 0 \\ w \neq 0 }} \frac{|F(t,w,\bar w)|}{|w|^2}  = 0,
\]
where $t=(t_1,\hdots,t_{n-2})$.
\end{definition}

First, we look at $M^3 \subset \C^3$. Let $p\in M$ be a hyperbolic point.
We have holomorphic coordinates $z_1,z_2,z_3$ of $\C^3$
and smooth local parametrization $(t, u, v)$ of $M$ such that
$M$ is represented centered at $p$ in the normal form (\ref{eqn_bishop_normal_form}):
\begin{equation}
  \label{eqn:hyper_normal}
  \begin{split}
    z_1 &= t + ih(t, u, v), \\   % DONE change x to t
    z_2 &= u + iv, \\
    z_3 &= w \bar w +  \gamma ( w^2 + \bar w^2 ) + F(t, w, \bar w), %\\
    %&= u^2 + v^2 + 2\gamma (u^2 -v^2) + F(t,u,v),
  \end{split}
\end{equation}
where $w=u+iv$, $\gamma > \frac{1}{2}$, $(t,u,v)$ are real parameters,
$h: \R^3 \to \R$ is a real-valued smooth function that vanishes to order two at zero
and $F: \R^3 \to \C$ is a complex-valued smooth function that vanishes to order three at zero.

% FIX: give general dimensions and add bounds

\begin{remark}
In the normal form, $F$ vanishes to order three at zero.
Therefore, its Taylor expansion is of the form
$$
  F(t,w, \bar w) =\sum_{j=3}^{k-1} \sum_{a+b+c=j} C_{a,b,c} t^a w^b \bar w^c + F_k(t,w, \bar w),
$$
where $F_k$ vanishes to order $k$ at zero and $C_{a,b,c}$ are constants, i.e.,
$$
\lim_{(t,w)\to 0} \frac{|F_k(t,w, \bar w)|}{\norm{(t,w, \bar w)}^k} = 0.
$$
\end{remark}

\begin{remark}
If $F$ vanishes to order two in $w$ at zero, then its
Taylor expansion near zero has the form
$$
  F(t,w, \bar w) =\sum_{j=3}^{k-1} \sum_{\substack{{a+b+c=j} \\ b+c\geq 2}} C_{a,b,c} t^a w^b \bar w^c + F_k(t,w, \bar w),
$$
where $F_k$ vanishes to order $k$ near zero and vanishes to order two in $w$.
Therefore, $F$ omits terms of the form  $t^j w, t^j \bar w$ and $t^j$.
\end{remark}

\begin{proof}[Proof of Theorem \ref{thm_hyperbolic_higher_dimension_vanish_order_two}]

% TODO: give some reasoning for higher dimension

We will consider $n=3$ for simplicity. The same proof can be generalized to any $n\geq 3$. For this consider $t = (t_1, ..., t_n) \in [-T,T]^{n-2}$.

Consider the proper holomorphic function $P:\C^3 \to \C^3$ given by
\[
  P(\zeta_0, \zeta_1, \zeta_2) = \left(\zeta_0, \zeta_1 + i \zeta_2, \zeta_1^2 + \zeta_2^2 + 2\gamma (\zeta_1^2 - \zeta_2^2) \right).
\]
We can solve for $P^{-1}(M)$ in terms of the parameters $t, w = u + iv$ by setting 
\[
  P(\zeta_0, \zeta_1, \zeta_2) = \left( t + ih(t, u, v), u + iv, u^2 + v^2 + 2 \gamma (u^2 - v^2) + F(t, u, v) \right)
\]
for $t,u,v$ real. The first coordinate gives us $\zeta_0 = t + ih(t, u, v)$. The second coordinate gives the relation
$\zeta_1 + i\zeta_2 = u+iv $. The last coordinate gives $\zeta_1^2 + \zeta_2^2 + 2\gamma (\zeta_1^2 - \zeta_2^2) = u^2 + v^2 + 2 \gamma (u^2 - v^2) + F(t, u, v)$.
The last two equations give a quadratic equation in $\zeta_1$.
Solving the quadratic equation yields:
\begin{equation}
  \begin{split}
    \zeta_1 &= \frac{(2\gamma - 1)(u+iv)}{4\gamma} \pm \frac{\sqrt{((2\gamma+1)u - (2\gamma-1)iv)^2 + 4\gamma F(t,u,v) }}{4\gamma}, \\
    \zeta_1 &= \frac{(2\gamma - 1)w \pm \sqrt{( 2\gamma \bar w + w )^2 + 4\gamma F(t,w, \bar w) }}{4\gamma}.
  \end{split}
\end{equation}
When $w\neq 0$, we can simplify the discriminant as follows:
\begin{equation}
  \begin{split}
    \sqrt{( 2\gamma \bar w + w )^2 + 4\gamma F(t,w, \bar w) } 
    = \sqrt{( 2\gamma \bar w + w )^2 \left( 1+ \frac{4\gamma F(t,w, \bar w)}{( 2\gamma \bar w + w )^2 } \right)}.
    %\sqrt{\left((2\gamma+1)u - (2\gamma-1)iv\right)^2 \left( 1 +  \frac{4\gamma F(t,u,v)}{((2\gamma+1)u - (2\gamma-1)iv)^2} \right) } \\
    %= \pm \left((2\gamma+1)u - (2\gamma-1)iv\right) \sqrt{ 1 +  \frac{4\gamma F(t,u,v)}{((2\gamma+1)u - (2\gamma-1)iv)^2}  } \\
  \end{split}
\end{equation}
We have two solutions corresponding to each branch as
\begin{equation}
  \begin{split}
    +\sqrt{D} = + ( 2\gamma \bar w + w ) \sqrt{\left( 1+ \frac{4\gamma F(t,w, \bar w)}{( 2\gamma \bar w + w )^2 } \right)}, \\
    -\sqrt{D} = - ( 2\gamma \bar w + w ) \sqrt{\left( 1+ \frac{4\gamma F(t,w, \bar w)}{( 2\gamma \bar w + w )^2 } \right)}. \\
  \end{split}
\end{equation}

As $F$ vanishes to order two in $w$ at zero and 
  $\gamma > \frac{1}{2}$, we have 
$$
  \lim_{\substack{(t,w)\to 0 \\ w\neq 0 }} \frac{|4\gamma F(t,w, \bar w)|}{|( 2\gamma \bar w + w )|^2 }
  \leq
  \lim_{(t,w)\to 0} \frac{|4\gamma F(t,w, \bar w)|}{|w|^2 (2\gamma - 1)^2 } = 0.
$$

  So $\frac{4\gamma F(t,w, \bar w)}{( 2\gamma \bar w + w )^2 }$ is small
  in a neighbourhood of zero. 
  Hence, $\sqrt{D}$ is well-defined as it is a compostion of $\sqrt{1+z}$
  which is well-defined near $0\in \C$
  with $\frac{4\gamma F(t,w, \bar w)}{( 2\gamma \bar w + w )^2 }$ which is small
  near $(t,w)=0$.

% 
% We can expand the square root with the expansion $(1+x)^{1/2} = 1 + x/2 - x^2/8 + x^3/16 + \hdots$
% 
% \begin{equation}
%   \begin{split}
%     \sqrt{ 1 +  \frac{4\gamma F(x,u,v)}{((2\gamma+1)u - (2\gamma-1)iv)^2}   } &= \\
%      1 + \frac{2\gamma F(x,u,v)}{((2\gamma+1)u - (2\gamma-1)iv)^2}
%     - &  \frac{2 \gamma F^2(x,u,v)}{((2\gamma+1)u - (2\gamma-1)iv)^4} + \hdots
%   \end{split}
% \end{equation}
% 
% \begin{equation}
%   \begin{split}
%     \pm \left((2\gamma+1)u - (2\gamma-1)iv\right) \sqrt{ 1 +  \frac{4\gamma F(x,u,v)}{((2\gamma+1)u - (2\gamma-1)iv)^2}  } = \\
%     \pm \Bigl( ((2\gamma+1)u - (2\gamma-1)iv) + \frac{2\gamma F(x,u,v)}{((2\gamma+1)u - (2\gamma-1)iv)^1} \\
%     -   \frac{2 \gamma^2 F^2(x,u,v)}{((2\gamma+1)u - (2\gamma-1)iv)^3} + \hdots \Bigr)
%   \end{split}
% \end{equation}
%We see that the second term onwards vanish to order two at zero, as the first partials in $u,v,x$ vanish.
%Set these terms to $f(x,u,v)$.
%\[
%  f(x,u,v) = \frac{2\gamma F(x,u,v)}{((2\gamma+1)u - (2\gamma-1)iv)^1}
%    -   \frac{2 \gamma^2 F^2(x,u,v)}{((2\gamma+1)u - (2\gamma-1)iv)^3} + \hdots
%\]

Define $f(t,w, \bar w)$ as
$$
f(t,w, \bar w) =     \sqrt{( 2\gamma \bar w + w )^2 + 4\gamma F(t,w, \bar w) }  - (2\gamma \bar w + w).
$$
\newcommand{\tempA}{( 2\gamma \bar w + w )}
It follows that $f$ is well-defined, and
we can show that
$$
\lim_{\substack{(t,w)\to 0 \\ |w| \neq 0}} \frac{f(t,w, \bar w)}{w} = 0,
$$
and show that $f$ is $\mathscr{C}^1$-smooth function with its first partial derivative vanishing at zero.
% TODO Show $f$ is C^1 fns

% If $F$ vanishes to order two,
% then we can show $f$ to be Lipschitz continuous at $0$.
% When $F$ vanishes to order three in $w$ at zero,
% we can check that $f$ is a $\mathscr{C}^1$-function
% with its first partial vanishing at zero.
% The rest of the proof follows if $f$ is Lipschitz continuous at $0$, but, for simplicity,
% we will assume $f$ to be $\mathscr{C}^1$-smooth
% function with its first partial vanishing at zero.

Now write $\zeta_1$ in terms of $f$:
\[
  \zeta_1 =
 \frac{(2\gamma - 1)(u+iv) \pm \left((2\gamma+1)u - (2\gamma-1)iv + f(t,u,v)\right) }{4\gamma}.
\]

% TODO: fix formatting for the long statment below
% 			t,u,v real looks ugly
%       make x -> t
We can rewrite the solutions in terms of $f$ as
\begin{equation*}
  \begin{split}
		& P^{-1}(M) = \left\{\left(
    t+ih(t,u,v), 
    u + \frac{f(t,u,v)}{4\gamma},
    v + \frac{i f(t,u,v)}{4\gamma}
    \right): t,u,v \text{ real} \right\} \\
    & \bigcup 
    \Bigl\{\Bigl(
    t+ih(t,u,v), 
    \frac{-u + (2\gamma - 1)iv}{2\gamma}  -  \frac{f(t,u,v)}{4\gamma}, 
    \frac{-(2\gamma -1) iu + v}{2\gamma}  -  \frac{if(t,u,v)}{4\gamma}
    \Bigr) : t,u,v \text{ real} \Bigr\}.
  \end{split}
\end{equation*}
Denote the first set of solutions as $V_1$ and the second set of solutions as $V_2$.
Then 
\begin{equation*}
  \begin{split}
    V1 &=
      \left\{
      \left( t, u, v \right)
      + \left( ih(t,u,v), \frac{f(t,u,v)}{4\gamma}, \frac{i f(t,u,v)}{4\gamma} \right)
    : t,u,v \text{ real} \right\},
    \\
    V2 &=  \Bigl\{ \Bigl(
          t,
          \frac{-u + (2\gamma - 1)iv}{2\gamma},
          \frac{-(2\gamma -1) iu + v}{2\gamma}
        \Bigr)  \\
        &+ \Bigl(
          ih(t,u,v), 
          - \frac{f(t,u,v)}{4\gamma},
          -  \frac{if(t,u,v)}{4\gamma}
        \Bigr)
        : t,u,v \text{ real} \Bigr\}.
  \end{split}
\end{equation*}
The solutions are $\mathscr{C}^1$ perturbations of planes. 
We can check that the two planes are totally real
by checking that any
basis for each plane is $\C$-linearly independent.
This gives us local polynomial convexity of $V_1$ and $V_2$ at zero as they
are $\mathscr{C}^1$-small perturbation of totally real planes (Theorem \ref{thm_lipschitz_graph_general_polynomial_convexity}).
We have $r_1>0$ and $r_2>0$,
such that $K_1 = V_1 \cap \overline B_{r_1}(0)$ 
and $K_2 = V_2 \cap \overline B_{r_2}(0)$ are polynomially convex.
We will show local polynomial convexity of $K_1 \cup K_2$.
For this, we use Kallin's lemma (Theorem \ref{thm_kallin_lemma}) with the following polynomial:
\[
  Q(\zeta_0, \zeta_1, \zeta_2) = \epsilon (\zeta_1^2 + \zeta_2^2) + i \zeta_1 \zeta_2.
\]
Choose $\epsilon > 0$ small enough so that
$ \epsilon \left(-1 + \frac{1}{\gamma}\right) - \frac{1}{2 \gamma} + \frac{1}{4 \gamma^{2}} < 0 $.
We can ensure this as $\gamma > 1/2$.

We can compute $Q(K_1)$ and represent it as
\begin{equation}
  \label{eqn:hyper_Q_k_1}
  \begin{split}
    Q(K_1) &= \{
      \epsilon u^2 + \epsilon v^2 + \frac{2\epsilon-1}{4\gamma} uf(t,u,v) - \frac{f^2(t,u,v)}{16\gamma^2}  \\
    &+ i\frac{2\epsilon+1}{4\gamma} vf(t,u,v) + i uv
    :  (t,u,v) \in U_1
  \},
  \end{split}
\end{equation}
where $U_1 \subset \R^3$ is a compact neighbourhood of zero.

As $f/w$ vanishes at zero, we have $vf/(u^2+v^2)$, $uf/(u^2+v^2)$ and $f^2/(u^2+v^2)$ converge to zero as $(t,u,v) \to 0$.
Hence, $\Re Q(K_1)$ is non negative, choosing $r_1>0$ smaller if necessary.
For
$$
 s = \left( t+ih(t,u,v), 
  u + \frac{f(t,u,v)}{4\gamma},
  v + \frac{i f(t,u,v)}{4\gamma} \right)
  \in K_1,
$$
if $\Re Q(s) = 0$, then $u=0$ and $v=0$,
and hence, $Q(s) = 0$. Analogously, for $Q(K_2)$:
\begin{equation}
  \begin{split}
    Q(K_2) &= \{ \\
    &+ u^{2} \left(\epsilon \left(-1 - \frac{1}{\gamma}\right) - \frac{1}{2 \gamma} - \frac{1}{4 \gamma^{2}}\right) + v^{2} \left(\epsilon \left(-1 + \frac{1}{\gamma}\right) - \frac{1}{2 \gamma} + \frac{1}{4 \gamma^{2}}\right)  \\
    &- \frac{f^{2}}{16 \gamma^{2}} + f \left(- \frac{\epsilon u}{2 \gamma} - \frac{u}{4 \gamma} - \frac{u}{4 \gamma^{2}}\right) \\
    &+ i \left(\epsilon \left(- \frac{f v}{2 \gamma} - \frac{2 u v}{\gamma}\right) + u v + \frac{f v}{4 \gamma} - \frac{f v}{4 \gamma^{2}} - \frac{u v}{2 \gamma^{2}}\right) \\
    &: (t,u,v) \in U_2
    \},
  \end{split}
\end{equation}
where $U_2 \subset \R^3$ is a compact neighbourhood of zero.
The coefficients of $u^2$ and $v^2$ are negative. Take $r_2>0$ 
smaller if necessary, to ensure $\Re Q(K_2)$ is non positive.
For $$s =
\left(
    t+ih(t,u,v), 
    \frac{-u + (2\gamma - 1)iv}{2\gamma}  -  \frac{f(t,u,v)}{4\gamma},
    \frac{-(2\gamma -1) iu + v}{2\gamma}  -  \frac{if(t,u,v)}{4\gamma}
\right) \in K_1,
$$
if $\Re Q(s) = 0$,
then $u=0$ and $v=0$, and hence, $Q(s) = 0$.

The image $Q(K_1)$ lies in the right half plane except at $0$,
and $Q(K_2)$ lies in the left half plane except at $0$.
We can conclude that $\widehat{Q(K_1)} \cap \widehat{Q(K_2)} = \{0\}$.

Finally, we wish to check the polynomial convexity of $Q^{-1}(0) \cap (V_1 \cup V_2)$.
\[
  Q^{-1}(0) \cap K_1 = \{ s \in K_1: Q(s) = 0 \}.
\] % TODO: complete sentence
With reference to equation (\ref{eqn:hyper_Q_k_1}), this happens only at $s \in K_1$
where $u=0$ and $v=0$. That is,
\[
  Q^{-1}(0) \cap K_1  \subset
\{ ( t+ih(t,0,0), 0, 0) :  (t,0,0) \in U_1  \}.
\]
Since $f$ vanishes in $w$, we have $f(t,0,0) = 0$.
Similarly,
\[
  Q^{-1}(0) \cap K_2 \subset \{ ( t+ih(t,0,0), 0, 0) :  (t,0,0) \in U_2 \}.
\]
We have $Q^{-1}(0) \cap (K_1 \cup K_2) = \{ ( t+ih(t,0,0), 0, 0) : (t,0,0) \in U_1\cup U_2 \}$
which is a $\mathscr{C}^1$-small perturbation of a totally real set.
Hence, $Q^{-1}(0) \cap (K_1 \cup K_2)$ is polynomially convex by
choosing $r_1>0$ and $r_2>0$ smaller if necessary to ensure this.

By Kallin's lemma (Theorem \ref{thm_kallin_lemma}),
we conclude that $K_1 \cup K_2$ is polynomially convex.
Further, as $\mathscr{P}(K_j) = \mathscr{C}(K_j)$ for $j=1,2$,
we have $\mathscr{P}(K_1 \cup K_2) = \mathscr{C}(K_1 \cup K_2)$.
So, every compact subset of $K_1 \cup K_2$ is polynomially convex.

Now we can assess the local polynomial convexity of $M$ at $p$. 
We see that $P^{-1}(M \cap \overline B(0,\epsilon)) \subset (K_1 \cup K_2)$
for some small $\epsilon > 0$.
This is true as we can show that $P$ defines
a homeomorphism for $V_j$ and $M$.

So $P^{-1}(M \cap \bar B(0,\epsilon))$ is polynomially convex and satsifies
$$
\mathscr{P}(P^{-1}(M \cap \bar B(0,\epsilon)))
=
\mathscr{C}(P^{-1}(M \cap \bar B(0,\epsilon))).
$$
Hence, by Theorem \ref{thm_proper_holo_mapping},
we can conclude that $M \cap \bar B(0,\epsilon)$
is polynomially convex and satisfies
$$
\mathscr{P}(M \cap \bar B(0,\epsilon))
=
\mathscr{C}(M \cap \bar B(0,\epsilon)).
$$
In particular, $M$ is locally polynomially convex at $p$.

\end{proof}

\begin{remark}
The condition that $F$ vanishes to order two in $w$ is not necessary.
  The following manifold $M^3 \subset \C^n$:
  \begin{equation}
    \begin{split}
      z_1 &= t, \\
      z_2 &= u+iv = w,\\
      z_3 &= w \bar w + \gamma (w^2 + \bar w^2) + t^2 \bar w,
    \end{split}
  \end{equation}
  can be shown to be locally polynomially convex at $0$ using another method:
  consider the polynomial $f:M\to \R$ as $f(z_1,z_2,z_3) = z_1$.
  The continuous function $f$ is real-valued on $M$ and is a polynomial.
  We can show that $f^{-1}(t)$ is polynomially convex as it
  can be identified with $M^2 \subset \C^2$ hyperbolic point for $t\in (-\epsilon, \epsilon)$  for $\epsilon>0$ small.
	By the Theorem \ref{thm_fibre_theorem}, $M$ is locally polynomially convex at $0$.
\end{remark}
We will generalize the above in Section \ref{section:flat-hyperbolic-point}.

\section{Hyperbolic points in $M^2 \subset \C^2$}
\label{section_hyperbolic_m2_c2}

We can represent the manifold $M \subset \C^2$ in Theorem (\ref{thm_polynomial_convexity_two_dim_with_C2_condition}) by the following
\begin{equation}
	\label{eqn_defining_equation_of_M_2}
	\begin{split}
	z_1 &= w \\
	z_2 &= w \bar w + \gamma (w^2 + \bar w^2) + F(w, \bar w),
	\end{split}
\end{equation}
where $\gamma > \frac{1}{2}$,
$w \in \overline{D_R(0)} \subset \C$,
and a smooth function $F: \overline{D_R(0)} \to \C$ 
that vanishes to order three at $0\in M$.

The following proof is inspired by the work of \textcite{forstnerivc1991new}.

\begin{proof}[Proof of Theorem \ref{thm_polynomial_convexity_two_dim_with_C2_condition}]
	Consider $P:\C^2 \to \C^2$ a proper holomorphic map given by:
	$$
	P(z_1,z_2) = \left(z_1, z_1 z_2 + \gamma(z_1^2 + z_2^2) \right)
	$$

  As shown in \cite[VI. Theorem]{forstnerivc1991new}, the pre-image $P^{-1}(M)$ can be seen as the union 
	$ S_1 \cup S_2 $ where
	\begin{equation*}
		\begin{split}
			S_1 &= \{ (\zeta, \bar \zeta + f(\zeta, \bar \zeta)): \zeta \in \overline{B_r(0)} \} \label{eqn:S_1-and-S_2}, \\
			S_2 &= \left\{ \left(\zeta, -\frac{1}{\gamma}\zeta - \bar \zeta + g(\zeta, \bar \zeta) \right): \zeta \in \overline{B_r(0)} \right\}, 
		\end{split}
  \end{equation*}
	and $f: \overline{B_r(0)} \to \C$,
	$g: \overline{B_r(0)} \to \C$ are $\mathscr{C}^1$-smooth functions
	with $df(0) = dg(0) = 0$.

  \subsection{Polynomial convexity of $S_1$ and $S_2$} 
  We will show $f$ and $g$ are Lipschitz to apply Theorem \ref{thm_lipschitz_graph_general_polynomial_convexity}.

	As $P(S_1) \subset M$,
	$ P(\zeta, \bar \zeta + f(\zeta, \bar \zeta)) \in M $
	for $\zeta \in \overline{B_r(0)}$.
	So $P(\zeta, \bar \zeta + f(\zeta, \bar \zeta))$ must satisfy the
	defining equation (\ref{eqn_defining_equation_of_M_2}) of $M$.
	Solving this yields:
	\begin{gather*}
		\zeta\bar\zeta + \gamma(\zeta^2 + \bar\zeta^2) + F(\zeta, \bar \zeta)
		= \zeta( \bar \zeta + f(\zeta, \bar \zeta)) + \gamma (\zeta^2 + (\bar \zeta + f(\zeta, \bar \zeta))^2)
	\end{gather*}
	which is a quadratic equation in $f(\zeta, \bar \zeta)$:
	\begin{equation}
		\label{eqn_quadratic_for_f}
	0 = \gamma f(\zeta, \bar \zeta)^2 +	(\zeta + 2 \gamma \bar\zeta ) f(\zeta, \bar \zeta)
	- F(\zeta, \bar \zeta).
	\end{equation}

	Solving for $f(\zeta, \bar \zeta)$ yields:
	$$
	f(\zeta, \bar \zeta) = - \frac{1}{2\gamma}\left( (\zeta + 2 \gamma \bar\zeta) - \sqrt{(\zeta+2\gamma\bar\zeta)^2 + 4\gamma F(\zeta, \bar \zeta)} \right).
	$$
  Here we choose the root that ensures $df(0)=0$.

	To make sense of the square root, consider the following inequality:
	\begin{align*}
    \left| \frac{4\gamma F(\zeta, \bar \zeta)}{(\zeta+2\gamma\bar\zeta)^2}\right| 
		& \leq 4\gamma  \left| \frac{F(\zeta, \bar \zeta)}{(\zeta+2\gamma\bar\zeta)^2}\right|
    \leq \frac{4\gamma}{(2\gamma - 1)^2} \left| \frac{F(\zeta, \bar \zeta)}{\zeta^2}\right|
		\leq \frac{4\gamma}{(2\gamma - 1)^2}  8 \norm{F}_{\mathscr{C}^2} \\
		&\leq \frac{4\gamma}{(2\gamma - 1)^2}  8 \left(\frac{(2\gamma-1)^3}{2^{14} \gamma^3}\right)
    \leq \left(\frac{(2\gamma-1)}{2^{9} \gamma^2}\right)
    < \frac{1}{4}.
	\end{align*}
  Here, we have used Lemma \ref{lemma_C2_inequality}, equation (\ref{eqn_theorem-C2-bound-main}),
	and the assumption $\gamma > \frac{1}{2}$.
  % Here, we have used lemma \ref{lemma_C2_inequality}, equation (\ref{eqn_theorem-C2-bound-main}),
  % and $\frac{2\gamma-1}{16\gamma^2} < 1$ for $\gamma \in (\frac{1}{2}, +\infty)$ by checking the absolute maximum.

  This function is well defined as we take
	the square root over the principal branch in the inequality above.
  \begin{align}
	f(\zeta, \bar \zeta) &= - \frac{(\zeta + 2 \gamma \bar\zeta)}{2\gamma}\left( 1 - \sqrt{1 + \frac{4\gamma F(\zeta, \bar \zeta)}{(\zeta+2\gamma\bar\zeta)^2}} \right)
  \nonumber
  \\
	\left|\frac{f(\zeta, \bar \zeta)}{\zeta + 2\gamma\bar\zeta}\right| &= 
	\frac{1}{2\gamma}\left| 1 - \sqrt{1 + \frac{4\gamma F(\zeta, \bar \zeta)}{(\zeta+2\gamma\bar\zeta)^2}} \right|
  \label{eqn_equation_f_by_gamma_zeta}
  \end{align}

  By Lemma \ref{lemma_sqrt_inequality},
	$$
	\left|\frac{f(\zeta, \bar \zeta)}{\zeta + 2\gamma\bar\zeta}\right| \leq
	\frac{1}{2\gamma} \frac{10}{11} \left| \frac{4\gamma F(\zeta, \bar \zeta)}{(\zeta+2\gamma\bar\zeta)^2}\right|
  \leq 
  \frac{2}{(2\gamma - 1)^2}  4 \norm{F}_{\mathscr{C}^2} .
	$$
  In particular, using equation (\ref{eqn_theorem-C2-bound-main}):
	$$
  \left|\frac{f(\zeta, \bar \zeta)}{\zeta + 2\gamma\bar\zeta}\right| \leq \frac{1}{8\gamma},
	$$
	and
  \[
    \left|\frac{F_{z}(\zeta, \bar \zeta)}{\zeta + 2\gamma\bar\zeta}\right| \leq
    \frac{1}{(2\gamma-1)} \left|\frac{F_{z}(\zeta, \bar \zeta)}{\zeta}\right|
    \leq
    \frac{1}{(2\gamma-1)} 4 \norm{F}_{\mathscr{C}^2}.
  \]

	Applying $\partial/\partial z$ to equation (\ref{eqn_quadratic_for_f}) and solving for $f_z$ gives
	$$
	f_{z}(\zeta, \bar\zeta) = \frac{F_z(\zeta, \bar \zeta) - f(\zeta, \bar\zeta)}{2 \gamma f(\zeta, \bar\zeta) + (\zeta + 2\gamma \bar\zeta)}
  =
	\frac{\frac{F_z(\zeta, \bar \zeta)}{\zeta+2\gamma\bar\zeta} - \frac{f(\zeta, \bar\zeta)}{\zeta+2\gamma\bar\zeta}}
	{2 \gamma \frac{f(\zeta, \bar\zeta)}{\zeta+2\gamma\bar\zeta} + 1}.
	$$
	$$
	|f_{z}(\zeta, \bar\zeta)| \leq
	\frac{ \left|  \frac{F_z(\zeta, \bar \zeta)}{\zeta+2\gamma\bar\zeta} \right| + \left| \frac{f(\zeta, \bar\zeta)}{\zeta+2\gamma\bar\zeta}\right| }
	{1 - 2 \gamma \left|\frac{f(\zeta, \bar \zeta)}{\zeta+2\gamma\bar\zeta}\right| }.
	$$

	Using the above inequalities, we obtain:
	$$
	|f_{z}(\zeta, \bar\zeta)| \leq
	\frac{ 	\frac{4}{(2\gamma-1)} \norm{F}_{\mathscr{C}^2} + 	\frac{8}{(2\gamma -1)^2} \norm{F}_{\mathscr{C}^2} }
	{1 - 2 \gamma \frac{1}{8\gamma} }
	\leq \frac{16}{3}  \left( 	\frac{2\gamma  +1}{(2\gamma -1)^2} \right) \norm{F}_{\mathscr{C}^2} 
	$$

  \[
    \leq 	\frac{2^5 \gamma}{(2\gamma -1)^2}  \norm{F}_{\mathscr{C}^2} 
    \leq \frac{2^5 \gamma}{(2\gamma -1)^2}  \frac{(2\gamma-1)^3}{2^{14} \gamma^3}
    \leq  \frac{(2\gamma-1)}{2^{9}\gamma^2}.
  \]

	Similarly, solving for $f_{\bar z}$ gives
	\[
	  f_{\bar z}(\zeta, \bar \zeta) = \frac{F_{\bar z}(\zeta, \bar \zeta) - 2\gamma f(\zeta, \bar\zeta)}{2 \gamma f(\zeta, \bar\zeta) + (\zeta + 2\gamma \bar\zeta)}.
  \]
	We obtain
  \[
	|f_{\bar z}(\zeta, \bar \zeta)| \leq
	\frac{16}{3} \left( \frac{6\gamma - 1}{(2\gamma -1)^2} \right) \norm{F}_{\mathscr{C}^2} 
  \leq  	\frac{2^5 \gamma}{(2\gamma -1)^2} \norm{F}_{\mathscr{C}^2} 
    \leq \frac{2^5 \gamma}{(2\gamma -1)^2}  \frac{(2\gamma-1)^3}{2^{14} \gamma^3}
    \leq  \frac{(2\gamma-1)}{2^{9}\gamma^2}.
  \]

  \newcommand{\uk}{\left(\frac{2\gamma-1}{32\gamma^2}\right)}

  Hence, $|f_{z}(\zeta, \bar \zeta)| \leq \frac{1}{2^4} \uk$
  and $|f_{\bar z}(\zeta, \bar \zeta)| \leq \frac{1}{2^4} \uk$
  for $\zeta \in \overline{B_r(0)}$.

	As $P(S_2) \subset M$,
	$ P(\zeta, -\zeta/\gamma - \bar\zeta + g(\zeta, \bar\zeta)) \in M $
	for $\zeta \in \overline{B_r(0)}$.
	So $P(\zeta, -\zeta/\gamma - \bar\zeta + g(\zeta, \bar\zeta))$  must satisfy
	the defining equation (\ref{eqn_defining_equation_of_M_2}) of $M$:
	\begin{equation*}
			\zeta\bar\zeta + \gamma(\zeta^2 + \bar\zeta^2) + F(\zeta, \bar \zeta)= \zeta (-\zeta/\gamma - \bar\zeta + g(\zeta, \bar\zeta)) + \gamma(\zeta^2 + (-\zeta/\gamma - \bar\zeta + g(\zeta, \bar\zeta))^2),
	\end{equation*}
	which is a quadratic equation in $g(\zeta, \bar\zeta)$:
	\begin{equation}
		0 = \gamma g(\zeta, \bar\zeta)^2 - (\zeta + 2 \gamma\bar\zeta) g(\zeta, \bar\zeta) - F(\zeta, \bar \zeta).
	\end{equation}
	Doing a similar analysis as above:
  $|g_{z}(\zeta, \bar\zeta)| \leq \frac{1}{2^4} \uk$
	and
  $|g_{\bar z}(\zeta, \bar\zeta)| \leq \frac{1}{2^4} \uk$
  for $\zeta \in \overline{B_r(0)}$.

  By Lemma (\ref{lemma_complex_meanvalue_theorem_lipschitz}), 
	\begin{equation}
		\label{eqn_lipshitz_ineqn_for_f_and_g}
		\begin{split}
			|f(\zeta_1, \bar \zeta_1) - f(\zeta_2, \bar\zeta_2)| \leq \uk | \zeta_1 - \zeta_2|, \\
			|g(\zeta_1, \bar\zeta_1) - g(\zeta_2, \bar\zeta_2)| \leq \uk | \zeta_1 - \zeta_2|.
		\end{split}
	\end{equation}
  We can check that $\uk < \frac{1}{2}$ for $\gamma > \frac{1}{2}$.
  By Proposition \ref{thm_lipschitz_graph_general_polynomial_convexity}, $S_1$ and $S_2$ are polynomially convex:
  for $S_1$, this follows immediately from the proposition. For $S_2$, we must do an additional biholomorphic change of coordinates
  $z_2 \mapsto z_2 + \frac{1}{\gamma} z_1$ before applying the proposition.

  \subsection{Polynomial convexity of $S_1 \cup S_2$}

	We will use Kallin's lemma (Theorem \ref{thm_kallin_lemma})
  to obtain the polynomial convexity of the union.

  Consider the polynomial $\psi: \C^2 \to \C$ given by
  $$
  \psi(z_1,z_2) = \frac{1}{4} (z_1^2 - z_2^2) + \left(\frac{2\gamma-1}{16\gamma^2}\right) z_1 z_2.
  $$

  \newcommand{\uep}{\left(\frac{2\gamma-1}{16\gamma^2}\right)}

  We will show $\Re \psi(S_1)$ is to the right of the y-axis and only touches the y-axis at zero
  and $\Re \psi(S_2)$ is to the left of the y-axis and only touches the y-axis at zero.

	From the Lipschitz inequalities in (\ref{eqn_lipshitz_ineqn_for_f_and_g}), we have
    $|f(\zeta, \bar\zeta)| \leq \uk |\zeta|$ and 
    $|g(\zeta, \bar\zeta)| \leq \uk |\zeta|$.
  Denote $\uk$ by $\alpha$, then $\uep = 2\alpha$.
  \begin{itemize}
    \item Compute $\Re \psi(S_1)$:
      let $(\zeta, \bar\zeta + f(\zeta, \bar\zeta)) \in S_1$.
			\begin{align*}
				\psi(\zeta, \bar\zeta + f(\zeta, \bar\zeta)) &= 
          \frac{1}{4} (\zeta^2 - (\bar\zeta + f(\zeta, \bar\zeta))^2) + 2\alpha ( |\zeta|^2 + \zeta f(\zeta, \bar\zeta) ) \\
					&= \frac{1}{4} (\zeta^2 - \bar\zeta^2 - f(\zeta, \bar\zeta)^2 - 2\bar \zeta f(\zeta, \bar\zeta) ) + 2\alpha |\zeta|^2 + 2\alpha \zeta f(\zeta, \bar\zeta).
			\end{align*}
      Taking the real part and applying triangle inequality, gives
			\begin{align*}
        \Re \psi(\zeta, \bar\zeta + f(\zeta, \bar\zeta))
				& \geq 2\alpha  |\zeta|^2  + \frac{1}{4} \Re (\zeta^2 - \bar\zeta^2) - \left| \frac{1}{4} ( - f(\zeta, \bar\zeta)^2 - 2\bar \zeta f(\zeta, \bar\zeta) ) + 2\alpha \zeta f(\zeta, \bar\zeta) \right| \\
				& \geq 2\alpha  |\zeta|^2  + 0 - \left| \frac{f(\zeta, \bar\zeta)^2}{4} \right|  
        - \left| \frac{\bar \zeta f(\zeta, \bar\zeta)}{2} \right| 
        - 2\alpha \left| \zeta f(\zeta, \bar\zeta) \right|  \\
				& \geq 2\alpha  |\zeta|^2 - \frac{1}{4} \alpha^2 |\zeta|^2 
        - \frac{1}{2} \alpha |\zeta|^2 
        - 2\alpha^2 | \zeta |^2
				%\geq |\zeta|^2  \left( 2\alpha  - \frac{1}{4} \alpha^2 - \frac{1}{2} \alpha - 2\alpha^2  \right)
        \geq  \frac{3}{8} \alpha |\zeta|^2.
			\end{align*}
      Hence, $\Re \psi(S_1)$ is to the right of the y-axis and touches the y-axis only at zero.

    \item Similarly, a computation for $\Re \psi(S_2)$ shows that
      % We can check that
      % $\Re \psi(\zeta, -\zeta/\gamma - \bar\zeta + g(\zeta, \bar\zeta)) \leq a_2 |\zeta|^2$
      % for $a_2<0$.
			\begin{align*}
				&\psi(\zeta, -\zeta/\gamma - \bar\zeta + g(\zeta, \bar\zeta)) \\
				&= \frac{1}{4} \left( \zeta^2 - \left( \zeta^2/\gamma^2 + \bar \zeta^2 + g^2(\zeta, \bar\zeta) + 2\zeta\bar \zeta/\gamma - 2\bar \zeta g(\zeta, \bar\zeta) - 2\zeta g(\zeta, \bar\zeta)/\gamma \right) \right)
				\\
				& + 2\alpha \left( - \zeta^2/2 - \zeta \bar \zeta + \zeta g(\zeta, \bar\zeta) \right).
			\end{align*}
      Taking the real part, and applying the triangle inequality
			\begin{align*}
				&\Re \psi(\zeta, -\zeta/\gamma - \bar\zeta + g(\zeta, \bar\zeta))
				\\
				& \leq  
        \frac{1}{4} \left| \frac{\zeta^2}{\gamma^2} \right|
        + \frac{1}{4} \left| g^2(\zeta, \bar\zeta) \right|
				+ \frac{1}{2} \left| \bar \zeta g(\zeta, \bar\zeta) \right|
				\\
				&+ \frac{1}{2} \left| \frac{\zeta g(\zeta, \bar\zeta)}{\gamma} \right|
        + \frac{2\alpha}{\gamma} |\zeta^2 | + 2\alpha | \zeta g(\zeta, \bar\zeta) |
        - \left(\frac{1}{2\gamma} + 2\alpha \right)|\zeta|^2
				\\
				& \leq  
				 |\zeta|^2 \left(
           \frac{\alpha}{8} + \frac{\alpha}{2} + \alpha + 4\alpha + \alpha - 2\alpha
          \right)
          + |\zeta|^2 \left(\frac{1-2\gamma}{4\gamma^2} \right)
				\\
				&\leq  |\zeta|^2 \left( 5\alpha + \frac{1-2\gamma}{4\gamma^2} \right)
        = |\zeta|^2 \left( 5\alpha - 8\alpha \right)
        = - |\zeta|^2 \left( 3\alpha \right).
			\end{align*}
      Hence, $\Re \psi(S_2)$ is to the left of the y-axis and touches the y-axis only at zero.
  \end{itemize}

  Finally, $\psi^{-1}(0)\cap (S_1 \cup S_2) = \{(0,0)\} \in \C^2$: if $(\zeta_1, \zeta_2) \in S_1$ and $\psi(\zeta_1, \zeta_2)=0$ 
  then $\Re \psi(\zeta_1, \zeta_2) =0$ and hence, $(\zeta_1,\zeta_2)=(0,0)$. Similarly, $\psi|_{S_2}^{-1}(0) = \{(0,0)\}$.
  Hence, by Kallin's lemma, $S_1 \cup S_2$ is polynomially convex and satisfies $\mathscr{P}(S_1 \cup S_2) = \mathscr{C}(S_1 \cup S_2)$.

	As $P: \C^2 \to \C^2 $ was a proper map, we have $M$ is polynomially convex and satisfies
  $\mathscr{P}(M) = \mathscr{C}(M)$.

\end{proof}

\section{Flat hyperbolic point}
\label{section:flat-hyperbolic-point}

\subsection{The set of CR-singularities}

The manifold $M$ defined in equation (\ref{eqn_flat_normal_form})
matches locally at zero with the solution of the following equations in $\C^n$ with coordinates $(z_1 = x_1 + i y_1, \hdots, z_n = x_n + i y_n)$:
\begin{equation}
  \begin{split}
    r_1 &= y_1 - f_1(x_1) = 0, \\
    \vdots \\
    r_{n-2} &= y_{n-2} - f_{n-2}(x_1, \hdots, x_{n-2}) = 0, \\
    r_{n-1} &= x_n - (2\gamma+1) x_{n-1}^2 - (2\gamma -1) y_{n-1}^2 - \Re F(x_1, \hdots, x_{n-2}, x_{n-1}, y_{n-1})=0, \\ %TO?: is it cool like this?
    r_{n} &= y_n - \Im F(x_1, \hdots, x_{n-2}, x_{n-1}, y_{n-1}) = 0,\\
  \end{split}
\end{equation}
where $R = (r_1, \hdots, r_n) : U\subset \C^{n} \to \R^n$ 
is smooth 
and $U$ is a neighbourhood of zero.

As shown in \textcite{webster1985}, consider the smooth function $B:M \to \C$ given by
$$
  \partial r_1 \wedge \hdots \wedge \partial r_n |_p = B(p)
  d z_1 \wedge \hdots \wedge d z_n |_p
$$
where $\partial = \sum_{j=1}^{n}  \frac{\partial}{\partial z_j} dz_j$. 
A point $p \in M$ has a complex tangent in $T_pM$ if and only if $B(p) = 0$.

We can compute $B$ at a point $z' = (x'_1 + i y'_1, \hdots,  x'_n + i y'_n) \in M \subset \C^n$ as
\[
  \begin{split}
    B(z') &= - \frac{1}{(2i)^{n-1}} \left((2\gamma+1) x_{n-1}' - i (2\gamma - 1) y_{n-1}' \right) + o(|z'|) \\
          &= - \frac{1}{(2i)^{n-1}} \left(2\gamma \bar z_{n-1}' + z_{n-1}' \right) + o(|z'|). 
  \end{split}
\]

From equation (\ref{eqn_flat_normal_form}),
consider the parametrization as a function $P: V \to \C$ where $V=[-T,T]^{n-2} \times \overline{D(r)}$.
The composition $B\circ P:V \to \C$ gives
\[
	B\circ P(t_1, \hdots, t_{n-2}, w, \bar w) = - \frac{1}{(2i)^{n-1}} \left(2\gamma \bar w + w \right) + O(2).
\]
Let $w = u+iv$, and compute the partial derivative \[
  \left. \frac{\partial^2 B}{\partial u \partial v} \right|_{0} = \frac{1}{(2i)^{2n-2}} ((2\gamma)^2 - 1) \neq 0 .
\]
By the Implicit Function theorem,
there is a smooth function $\eta: (-T,T)^{n-2} \to D(r)$ choosing $T,r$ smaller if necessary,
such that the zero set of $B\circ P$ is given by the $(n-2)$-parameter set
\[
  \{(t,\eta(t)) \in M : t\in [-T,T]^{k-2} \}.
\]

The set of CR-singularities for $M$ is given by the $(n-2)$-parameter set
\begin{equation}
  \{(t+i\tilde f(t), \eta(t), \phi(\eta(t))) \in M : t\in [-T,T]^{k-2} \},
\end{equation}
where $\phi(w) = w \bar w + \gamma ( w^2 + \bar w^2) + F(t_1, \hdots, t_{n-2},w, \bar w)$
and $\tilde f(t) = (f_1(t_1), \hdots, f_{n-2}(t_1, \hdots, t_{n-2}))$.

Observer that
\begin{equation}
  \label{eqn__del bar or phi}
	\left. \frac{\partial \phi}{\partial \bar w} \right|_{(t_1, \hdots, t_{n-2}, w)} =  (w + 2\gamma \bar w) + O(2) = - (2i)^{n-1} B\circ P(t_1, \hdots, t_{n-2}, w, \bar w) = 0
\end{equation}
Hence, at a point $p$ in $[B=0]$, we have $J P_{*} (\partial/\partial u) = P_{*} (\partial/\partial v)$ and hence, $P_{*} (\partial/\partial u) \in H_pM$, where $w=u+iv$.
It follows that the complex tangent is contained in the last two coordinates.
The dimension $\dim H_qM = 1$ 
for all $q$ in a neighbourhood of zero of $[B=0]$ because $0 \in M$ is a CR singularity of order one
and $\dim H_pM$ is an upper semi continuous function. Hence, $$\text{span}_\R\{P_{*} (\partial/\partial u), P_{*} (\partial/\partial v)\} = H_pM. $$
% DONE fix paragraph completely

\subsection{Slices of $M$}
For $k\geq 2$ and $t \in [-T,T]^{n-k}$,
define $M^k_{t'_1, \hdots, t'_k}$ as follows:
\begin{equation}
  \label{eqn:slice of M}
  \begin{split}
  z_1 &= t'_1 +i f_1(t'_1), \\
  \vdots \\
  z_{n-k} &= t'_{n-k} +i f_{n-k}(t'_1, \hdots, t'_{n-k}), \\
  z_{n-k+1} &= t_{n-k+1} +i f_{n-k+1}(t'_1, \hdots, t'_{n-k}, t_{n-k+1}), \\
  \vdots \\
  %z_{n-2} &= t_{n-2} +i f_{n-2}(t_1, \hdots, t_{n-2}), \\
  z_{n-2} &= t_{n-2} +i f_{n-2}(t'_1, \hdots, t'_{n-k}, t_{n-k+1}, \hdots, t_{n-2}), \\
  z_{n-1} &= w, \\ 
  z_n &= w \bar w + \gamma ( w^2 + \bar w^2) + F(t'_1, \hdots, t'_{n-k}, t_{n-k+1}, \hdots, t_{n-2}, w, \bar w), 
  \end{split}
\end{equation}
where $t_j \in [-T, T]$ for $j=n-k+1,\hdots,n-2$; $w \in \overline{ D_r(0)}$.
Then $M^k_{t'_1, \hdots, t'_k}$ is a $k$-dimensional smooth submanifold of $M$.

\subsection{Normal form representation of two-dimensional slices}
\label{section:normal form of two-dimensional slices}
Fix $t\in[-T,T]^{n-2}$,
we will parametrize $M_{t}^2$ centered at CR singularity $(t+i\tilde f(t),\eta(t), \phi(\eta(t))$
as per the normal form of a two-dimensional manifold in \cite{bishop1965differentiable}.

Note that $M_{t}^2$ can be realized as a manifold in $\C^2$ as the first $n-2$ coordinates are fixed.
The point $(\eta(t), \phi(\eta(t)) \in M_{t}^2$ is a complex point as seen in equation (\ref{eqn__del bar or phi}).

Write the Taylor expansion of $\phi_t(w) := \phi(t,w)$
about the complex point $\eta(t)$. Set $\zeta = w - \eta(t)$.
$$
\phi_t(w) = \beta^t_{0,0} + \beta^t_{1,0} \zeta + \beta^t_{1,1} \zeta \bar \zeta
+ \beta^t_{2,0} \zeta^2 + \beta^t_{0,2} \bar \zeta^2 + G_t(\zeta),
$$
where $w \in B_r(0)$,
$$
\beta_{i,j}^t =\left. \frac{\partial^{i+j} \phi_t}{\partial w^i \partial \bar w^j} \right|_{w=\eta(t)},
$$
and $G_t:B_r(-\eta(t)) \to \C$ is smooth and vanishes to order three in $\zeta$.
% Remark: use lemma 1.4 Tu - intro to mnflds (star shaped domain)

Then $M_{t}^2$ is parametrized by
\begin{align*}
  z_2 &= \zeta + \eta(t), \\
  z_3 &= \beta^t_{0,0} + \beta^t_{1,0} \zeta + \beta^t_{1,1} \zeta \bar \zeta
+ \beta^t_{2,0} \zeta^2 + \beta^t_{0,2} \bar \zeta^2 + G_t(\zeta),
\end{align*}
where $\zeta \in B_r(-\eta(t))$.

Consider the following biholomorphic change of coordinates:
$z_2 \to z_2 - \eta(t) $ and $z_3  \to z_3 - \beta_{0,0} - \beta_{1,0} z_2$.
In the new coordinates $M_t^2$ is parametrized by:
\begin{align*}
  z_2 &= \zeta, \\
  z_3 &= \beta^t_{1,1} \zeta \bar \zeta
+ \beta^t_{2,0} \zeta^2 + \beta^t_{0,2} \bar \zeta^2 + G_t(\zeta).
\end{align*}
Followed by biholomorphic coordinate change
% $z_3 \to (\beta^t_{0,2} - \beta^t_{2,0}) z_2^2$ and 
$z_3  \to \frac{z_3}{\beta^t_{1,1}}$ and
parameter change
$\zeta  \to  e^{i\theta} \zeta $ 
such that $(\beta^t_{0,2}/\beta^t_{1,1} ) e^{-i 2 \theta} > 0$.
Define $\gamma_t := (\beta^t_{0,2}/\beta^t_{1,1} ) e^{-i 2 \theta} = |\beta^t_{0,2}/\beta^t_{1,1}|$,
then another coordinate change 
$z_3 \to z_3 + (\gamma_t - \frac{\beta^t_{2,0}}{\beta^t_{1,1}}e^{i2\theta}) z_2^2$ gives:
\begin{align*}
  z_2 &= \zeta, \\
  z_3 &= \zeta \bar \zeta
  + \gamma_t ( \zeta^2 +  \bar \zeta^2 ) + \hat G_t(\zeta),
\end{align*}
where $\hat G_t(\zeta) = \frac{1}{\beta^t_{1,1}} G_t(e^{i\theta} \zeta)$
and $\zeta \in B_r(-e^{i\theta}\eta(t))$.
Here, $\hat G_t$  vanishes to order three in $\zeta$
and $\gamma_t = |\beta^t_{0,2}/\beta^t_{1,1}|$.

The $\mathscr{C}^2$-norm of $\hat G_t$ varies continuously with $t$ as
$$
G_t(\zeta) = 
\phi_t(w) - \beta^t_{0,0} - \beta^t_{1,0} \zeta - \beta^t_{1,1} \zeta \bar \zeta
- \beta^t_{2,0} \zeta^2 - \beta^t_{0,2} \bar \zeta^2 
$$
is a smooth function.

\subsection{Local polynomial convexity at a flat hyperbolic point}

\begin{proof}[Proof of Theorem \ref{thm_flat_CR_singularity}]

  Fix $t \in [-T,T]^{n-2}$.
  Consider the two-dimensional slice $M^2_t$ as shown in equation (\ref{eqn:slice of M}).

  As shown in Section \ref{section:normal form of two-dimensional slices},
  after a biholomorphic coordinate change we can represent this two-dimensional slice
  as
  \begin{align*}
    z_{n-1} &= \zeta, \\
    z_n &= \zeta \bar \zeta
    + \gamma_{t} ( \zeta^2 +  \bar \zeta^2 ) + \hat G_{t}(\zeta),
  \end{align*}
  where $\hat G_{t}$  vanishes to order three in $\zeta$. 

  Choose $r>0$ small enough such that
  $$
  \norm{\hat G_{0}}_{\mathscr{C}^2} \leq \frac{(2\gamma_{0} -1)^3}{2^{14} \gamma_{0}^3}.
  $$ We can do this since $\hat G_{0}$ vanishes to order three at zero.

  Notice that $G_{t}$ varies continuously with $t$.
  Choose $T>0$  small enough such that for $t\in [-T,T]^{n-2}$,
  $$
  \norm{\hat G_{t}}_{\mathscr{C}^2} \leq \frac{(2\gamma_{t} -1)^3}{2^{14} \gamma_{t}^3}.
  $$
  % DONE 2 dimensional case reference here
	By Theorem \ref{thm_polynomial_convexity_two_dim_with_C2_condition},
  each two-dimensional slice $M^2_t$ produced by fixing
  $t \in [-T,T]^{n-2}$ is polynomially convex and satisfies
  $\mathscr{P}(M^2_t) = \mathscr{C}(M^2_t)$.
  %We have seen that $M^2_t$ is polynomially convex for any $t\in[-T,T]^{n-2}$.
  Induction hypothesis: assume that for dimension $k \geq 2$ slices $M^k_{t'}$ are polynomially convex 
  for all $t'=(t'_1, \hdots, t'_{n-k}) \in [-T,T]^{n-k}$.
  We will show that $M^{k+1}_{t^o}$ is polynomially convex
  for all $t^o=(t'_1, \hdots, t'_{n-k-1}) \in [-T,T]^{n-k-1}$.

  Fix $t^o=(t'_1, \hdots, t'_{n-k-1}) \in [-T,T]^{n-k-1}$.
  The set $M^{k+1}_{t^o}$ is parametrized by

  \begin{equation}
    \begin{split}
    z_1 &= t'_1 +i f_1(t'_1), \\
    \vdots \\
    z_{n-k-1} &= t'_{n-k-1} +i f_{n-k}(t'_1, \hdots, t'_{n-k-1}), \\
    z_{n-k} &= t_{n-k} +i f_{n-k}(t'_1, \hdots, t'_{n-k-1}, t_{n-k}), \\
    z_{n-k+1} &= t_{n-k+1} +i f_{n-k+1}(t'_1, \hdots, t'_{n-k-1}, t_{n-k}, t_{n-k+1}), \\
    \vdots \\
    %z_{n-2} &= t_{n-2} +i f_{n-2}(t_1, \hdots, t_{n-2}), \\
    z_{n-2} &= t_{n-2} +i f_{n-2}(t'_1, \hdots, t'_{n-k}, t_{n-k+1}, \hdots, t_{n-2}), \\
    z_{n-1} &= w, \\ 
    z_n &= w \bar w + \gamma ( w^2 + \bar w^2) + F(t'_1, \hdots, t'_{n-k}, t_{n-k+1}, \hdots, t_{n-2}, w, \bar w). 
    \end{split}
  \end{equation}

  Consider the function $p_s : M^{k+1}_{t^o}  \to \C$ 
  defined by projection $z \mapsto z_{n-k}$.
  Let $A = p_{n-k}(M^{k+1}_{t^o})$. Then the fibers of $p_{n-k}$
  are given by
  $$
  p_{n-k}^{-1}(a) = M^{k}_{(t'_1, \hdots, t'_{n-k-1}, \Re a)}
  $$
  for $a\in A$. We know that $M^k_{(t'_1, \hdots, t'_{n-k-1}, \Re a)}$ is polynomially
  convex by the induction hypothesis.
  By Theorem \ref{thm_fibre_theorem}, we have $M^{k+1}_{t^o}$ is polynomially convex.

  Hence, the manifold $M^{n} = M$ is polynomially convex when $T,r$ are chosen sufficiently small.
  The induction also implies $\mathscr{C}(M) = \mathscr{P}(M)$.

\end{proof}

\newpage
\printbibliography
\end{document}